\documentclass[12pt,reqno]{amsart}

\newcommand\version{October 23, 2026}
\usepackage{amsmath,amssymb,amsthm,mathtools}
\usepackage{hyperref}

\newtheorem{theorem}{Theorem}[section]
\newtheorem{proposition}[theorem]{Proposition}
\newtheorem{lemma}[theorem]{Lemma}

\theoremstyle{definition}

\theoremstyle{remark}

\numberwithin{equation}{section}

\newcommand{\R}{\mathbb R}
\newcommand{\C}{\mathbb C}
\newcommand{\Z}{\mathbb Z}
\newcommand{\N}{\mathbb N}
\newcommand{\cb}{\mathcal B}
\renewcommand{\epsilon}{\varepsilon}
\renewcommand{\phi}{\varphi}

\newcommand{\J}{\mathfrak J}
\newcommand{\one}{\mathbf{1}}
\newcommand{\me}[1]{\mathrm{e}^{#1}}
\newcommand{\supp}{\operatorname{supp}}
\newcommand{\diag}{\operatorname{diag}}
\newcommand{\cn}{\mathcal N}
\newcommand{\cm}{\mathcal M}
\newcommand{\fh}{\mathfrak H}

\def\ch{\mathcal{H}}
\def\cx{\mathcal{X}}
\DeclareMathOperator{\im}{Im}

\DeclareMathOperator{\re}{Re}
\DeclareMathOperator*{\slim}{s-lim}

\begin{document}

\title[Multiplier theorem for $(-\Delta)^{\alpha/2}-\kappa/|x|^\alpha$ --- \version]{A H\"ormander spectral multiplier theorem for \\ fractional Hardy operators}

\author[K. Merz]{Konstantin Merz}
\address[Konstantin Merz]{Institute for Theoretical Physics, ETH Zurich, Wolfgang-Pauli-Strasse 27, 8093 Zurich, Switzerland, and Institut f\"ur Analysis und Algebra, Technische Universit\"at Braunschweig, Universit\"atsplatz 2, 38106 Braun\-schweig, Germany}
\email{k.merz@tu-bs.de, k.merz@phys.ethz.ch}


\date{\version}
\maketitle
\begin{abstract}
  We consider the fractional Laplacian with added critical or subcritical Hardy potential, called Hardy operator. We prove a H\"ormander spectral multiplier theorem, i.e., the boundedness of sufficiently smooth multipliers of these Hardy operators on Lebesgue spaces. The range of admissible Lebesgue exponents is optimal, but the smoothness assumption is probably not. The proof combines the Helffer--Sj\"ostrand formula with a quantitative Jaffard lemma for operator-valued matrices.
\end{abstract}


\section{Introduction and main result}
\label{s:intro}


\subsection{Setting}
Let $d\in\N=\{1,2,...\}$, $\alpha\in(0,d)\cap(0,2]$, and
\begin{align}
  \kappa_{\rm c} = \kappa_{\rm c}(d,\alpha) = \frac{2^\alpha\Gamma\left(\frac{d+\alpha}{4}\right)^2}{\Gamma\left(\frac{d-\alpha}{4}\right)^2}
\end{align}
be the sharp constant in the Hardy--Kato--Herbst inequality \cite{Kato1966,Herbst1977}, for short Hardy inequality, which states
\begin{align}
  (-\Delta)^{\alpha/2} \geq \kappa_{\rm c}|x|^{-\alpha}
\end{align}
in the sense of quadratic forms on $L^2(\R^d)$; see also \cite{Kovalenkoetal1981,Yafaev1999,Franketal2008H,FrankSeiringer2008,Bogdanetal2016,BogdanMerz2024} for different proofs. For $\kappa\leq\kappa_{\rm c}$, we consider the Friedrichs extension
\begin{align}
  H = (-\Delta)^{\alpha/2}-\frac{\kappa}{|x|^\alpha} \quad \text{in} \ L^2(\R^d)
\end{align}
of the corresponding expression on $C_c^\infty(\R^d\setminus\{0\})$. When $\kappa<\kappa_{\rm c}$, then its form domain is the inhomogeneous Sobolev space $H^{\alpha/2}(\R^d)$. Moreover, $C_c^\infty(\R^d\setminus\{0\})$ is a form core \cite[Theorem~5.4]{Bogdanetal2019}.
We call $\kappa/|x|^\alpha$ the Hardy potential and $H$ the Hardy operator. If $\alpha=1$, $H$ is an important model in relativistic quantum mechanics, where it describes, e.g., electrons of a heavy atom in the vicinity of the nucleus, see, e.g., \cite{Franketal2008,Franketal2020P}, or in the study of gravitational collapse of stars, see, e.g., \cite{Chandrasekhar1931,LiebThirring1984,LiebYau1987,LiebYau1988,FrohlichLenzmann2007D}. Note that $H$ has purely absolutely continuous spectrum $[0,\infty)$, see, e.g., \cite[Remark~5.5]{IshidaMerz2026} for a complete argument. However, $H$ possesses so-called generalized ground states, i.e., non-square-integrable functions $h$ for which $Hh=0$. One important example is $h(x)=|x|^{-\delta}$, where $\delta=\delta(d,\alpha,\kappa)\in(-M,(d-\alpha)/2]$ with $M=\infty$ if $\alpha=2$ and $M=\alpha$ if $\alpha<2$ is the unique solution of $\kappa=\Phi(\delta)$ with
\begin{align}
  \left(-M,\frac{d-\alpha}{2}\right] \ni \delta \mapsto \Phi(\delta) := \frac{2^\alpha\Gamma\left(\frac{\delta+\alpha}{2}\right) \Gamma\left(\frac{d-\delta}{2}\right)}{\Gamma\left(\frac\delta2\right)\Gamma\left(\frac{d-\alpha-\delta}{2}\right)}.
\end{align}
Note that $\Phi(0)=0$, $\Phi((d-\alpha)/2)=\kappa_{\rm c}$, and $\lim_{\delta\to-M}\Phi(\delta)=-\infty$. Moreover, $\Phi(\delta)$ is strictly monotonically increasing \cite{Franketal2008H}. Since our main result merely concerns $\alpha<2$, as the corresponding result for $\alpha=2$ is well-known as we recall below, we will from now on only consider $\alpha\in(0,\min\{2,d\})$ and denote
\begin{align}
  \label{eq:defdelta}
  \delta := \Phi^{-1}(\kappa) \in \left(-\alpha,\frac{d-\alpha}{2}\right].
\end{align}



In \cite{Merz2021}, a {\em H\"ormander spectral multiplier theorem} is proved for $H$ when $\alpha<2$ and $\kappa<0$, i.e., sufficiently smooth functions $F(H)$, initially defined by the spectral theorem in $L^2$, are $L^p$-bounded for certain $p\neq2$. The main result of this paper is an extension to $0<\kappa\leq\kappa_{\rm c}$. To state it, choose a bump function $\psi\in C_c^\infty((0,\infty))$ satisfying $\sum_{k\in\mathbb Z}\psi(2^{-k}\lambda)=1$ for all $\lambda>0$, a real-valued function $\chi\in C_c^\infty((0,\infty))$ which equals one on an open neighborhood of $\supp\psi$, and define for $s\geq0$ the seminorm
\begin{align}
  \label{eq:Hs}
  \fh_s(F) := \sup_{t>0}\|\chi(\cdot)F(t\,\cdot)\|_{H^s(\R)},
\end{align}
which is, modulo an irrelevant multiplicative factor, independent of the chosen cutoffs $\psi$ and $\chi$. Here, $H^s(\R^d)$ denotes the $L^2(\R^d)$-based Sobolev space of order $s$. Defining further
\begin{align}
  \label{eq:norm}
  \mathfrak N_s(F) := \sup_{\lambda\ge0}|F(\lambda)| + \fh_s(F),
\end{align}
then, according to \cite[Theorem~2]{Merz2021}, there is $s_{d,\alpha,\kappa}>0$ such that for all bounded, measurable $F:[0,\infty)\to\C$ with $\mathfrak N_s(F)<\infty$ for any $s>s_{d,\alpha,\kappa}$, one has $\|F(H)\|_{p\to p}\lesssim \mathfrak N_s(F)$ for all $\kappa<0$ and $1<p<\infty$. This statement is classical for $\kappa=0$ and goes back to H\"ormander \cite{Hormander1960}, which improves Mikhlin's multiplier theorem \cite{Mihlin1956} where $\mathfrak N_s(F)<\infty$ is replaced by the condition
\begin{align*}
  |\lambda^{j}\partial_\lambda^j F(\lambda)| \leq c_j, \quad j=0,...,\left\lfloor\frac{d}{2}\right\rfloor+1,
\end{align*}
for positive constants $c_j$. Classical examples of H\"ormander multipliers are imaginary powers $F(\lambda)=\lambda^{i\tau}$ (which are useful, e.g., in the context of complex interpolation) and the Bochner--Riesz means $F_R^\sigma(\lambda)=(1-\lambda/R)_+^\sigma$ for $R,\sigma>0$, which satisfy $\mathfrak N_s(F_R^\sigma)<\infty$ if and only if $\sigma>s-1/2$. The latter are smoothed versions of the ball multiplier (which is $L^p(\R^d)$-bounded if and only if $p\in(1,\infty)$ when $d=1$, and $p=2$ when $d\geq2$, see \cite{Fefferman1971}); a fundamental (and still widely open) question is to characterize the optimal range of $\sigma$ such that $F_R^\sigma$ is $L^p$-bounded. Finally, multiplier theorems are vital in constructing basic Littlewood--Paley theories, in particular in proving square function estimates and Bernstein inequalities, to advance the study of nonlinear PDEs.


\subsection{Main result}

To state the main result, we introduce, for given $\delta=\delta(\kappa)\in(-\alpha,(d-\alpha)/2]$ the exponents
\begin{align}
  \label{eq:range}
  \begin{split}
    & p_- := \max\left\{1,\frac{d}{d-\delta}\right\}, \qquad
      p_+ := \frac{d}{\max\{\delta,0\}}=p_-', \\
    & p_- < p_0 < 2, \qquad q = p_0' \in (2,d/\max\{\delta,0\}).
  \end{split}
\end{align}
Here, for $p\in(1,\infty)$, we write $p'=(1-1/p)^{-1}$. If $p=1$, then $p'=\infty$ and vice versa.
Thus, $p_-=1$ and $p_+=\infty$ if and only if $\kappa\leq0$. Finally, let us fix a, not necessarily optimal, regularity parameter
\begin{align}
  \label{eq:indices}
  S_* := \left\lceil 1+\frac{2(d+\alpha)}{\min\{1,\alpha\}}\right\rceil + 5.
\end{align}

Then, our main result can be stated as follows.

\begin{theorem}
  \label{thm:main}
  Let $d\in\N$, $0<\alpha<\min\{2,d\}$, $\kappa\leq\kappa_{\rm c}$, $s>S_*$, and let $F:[0,\infty)\to\C$ be bounded and measurable with $\mathfrak N_s(F)<\infty$. Then, for all $p_0\in(p_-,2)$, we have that $F(H)$ is of weak-type $(p_0,p_0)$, i.e., there is a constant $C_{d,\alpha,\kappa,p_0,s}>0$ such that
  \begin{align}
    \label{eq:weak}
    \|F(H)f\|_{L^{p_0,\infty}(\R^d)} \leq C_{d,\alpha,\kappa,p_0,s} \mathfrak N_s(F) \|f\|_{L^{p_0}(\R^d)}, \quad f\in L^{p_0}(\R^d).
  \end{align}
  Moreover, $F(H)$ extends uniquely to an $L^p(\R^d)$-bounded operator for all $p_-<p<p_+$ with
  \begin{align}
    \label{eq:strong}
    \|F(H)\|_{L^p(\R^d)\to L^p(\R^d)} \leq C_{d,\alpha,\kappa,p,s} \mathfrak N_s(F).
  \end{align}
\end{theorem}

The range of admissible $p$ is sharp for $\kappa>0$ as can be seen by taking $F(H)=\me{-H}$ and using the heat kernel bounds \eqref{eq:HK} below; see also \cite[p.~112]{Merz2021}.
Theorem~\ref{thm:main} is well-known for $\alpha=2$, see, e.g., \cite{Blunck2003,Killipetal2018}. In the latter paper it is used to characterize Sobolev spaces adapted to powers of $H$. The function space analysis of \cite{Killipetal2018} has been generalized to $\alpha<2$ in \cite{Franketal2021} for $p=2$, in \cite{Merz2021} for $p\neq2$ and $\kappa<0$, and eventually in \cite{BuiDAncona2023} for $p\neq 2$ and $\kappa\leq\kappa_{\rm c}$.
See also \cite{FrankMerz2023,BuiMerz2026} for similar conclusions for $(-\Delta)^{\alpha/2}|_{\R_+^d}-\kappa/x_d^\alpha$ and \cite{Buietal2026} for the Kolmogorov operator $(-\Delta)^{\alpha/2}-\kappa|x|^{-\alpha}x\cdot\nabla$.
Note that \cite{BuiNader2022} proves $L^p$-boundedness for Laplace-transform multipliers of $H$, including imaginary powers, also for all $p\in(p_-,p_+)$. Our Theorem~\ref{thm:main} here allows general bounded multipliers satisfying \eqref{eq:Hs}, without requiring a Laplace-transform representation.

Let us now abbreviate the notation appearing in Theorem~\ref{thm:main}. For functions $A,B \geq0$, we write $A\lesssim B$ to indicate that there is a constant $c$ such that $A\leq cB$. If $c$ depends on a parameter $\tau$, we may write $A\lesssim_\tau B$, but the dependence on fixed quantities like $d,\alpha,\kappa,\psi,\chi$ is usually ignored. The notation $A\sim B$ means $A\lesssim B \lesssim A$.
Displayed constants like $C$ in estimates may change from estimate to estimate, but we continue to use the same letter by an abuse of notation.
Furthermore, $\|A\|_\ch$ or $\|A\|$ denotes the operator norm of a linear operator on a given Hilbert space $\ch$. We denote the Lebesgue norm of a function $f$ by $\|f\|_{L^p(\R^d)}$ or, for short, by $\|f\|_{L^p}$ or even $\|f\|_p$. Likewise, we abbreviate $\|A\|_{L^p(\R^d)\to L^q(\R^d)}$ to $\|A\|_{L^p\to L^q}$ or even $\|A\|_{p\to q}$ for a linear operator $A$.
From now on, unless stated otherwise, every lemma, proposition, or theorem uses the assumptions in Theorem~\ref{thm:main}.


\subsection{Earlier results}

Given H\"ormander's theorem for Fourier multipliers it is natural to ask to which extent this theorem extends to more general operators and settings. Let us mention only some extensions: H\"ormander's theorem has, e.g., been extended to group invariant Laplace operators acting on Lie groups of polynomial growth \cite{Alexopoulos1994}, to homogeneous Laplace operators acting on homogeneous groups \cite{FollandStein1982,DeMicheleMauceri1987,MauceriMeda1990,Christ1991}, and to elliptic operators on compact manifolds, see, e.g., \cite{SeegerSogge1989,Seegeretal1991} and non-compact manifolds, see, e.g., \cite{ClercStein1974,ChristMueller1996,Guillarmouetal2013,Martinietal2018}.
Spectral multiplier theorems for arbitrary non-negative self-adjoint operators on Hilbert spaces $L^2(X)$, where $X$ is a doubling metric measure space, were, e.g., investigated in \cite{Duongetal2001,Duongetal2002} under the assumption that the associated heat kernels satisfy Gaussian upper heat kernel bounds. In this regard, see also the works \cite{BlunckKunstmann2003,Blunck2003,Blunck2007}, as well as recent improvements in \cite{Chenetal2016,Chenetal2020}.
The work \cite{Hebisch1990} proves a H\"ormander spectral multiplier theorem for $-\Delta+V(x)$ with $V(x)\geq0$ in $L^2(\R^d)$. A key idea there was to represent the given multiplier in terms of an integral involving the holomorphic extension of the (real-time) heat kernel to complex times with positive real part. The analysis then crucially relies on the Gaussian upper heat kernel bound for $-\Delta+V$ at real times since, by a Phragm\'en--Lindel\"of theorem, the complex-time heat kernel retains Gaussian decay. In \cite{Hebisch1995}, a H\"ormander multiplier theorem is proved for operators whose heat kernel only decays algebraically and satisfies a certain H\"older condition, which is analogous to the classical H\"ormander condition in singular integral theory. Prime examples for operators with slow heat kernel decay are fractional Schr\"odinger operators $(-\Delta)^{\alpha/2} + V$ with $\alpha\in(0,2)$ in $L^2(\R^d)$. The abstract theorem in \cite{Hebisch1995} was applied in \cite{Merz2021} to prove a H\"ormander multiplier theorem for the present operator $H=(-\Delta)^{\alpha/2}-\kappa/|x|^\alpha$ when $\kappa<0$; the assumptions in \cite{Hebisch1995} are not fulfilled if $\kappa>0$.
The above examples indicate that the subject of spectral multipliers is very broad, and we do not attempt to provide a comprehensive bibliography here. For further literature, we refer, e.g., to the above mentioned works and the references cited therein.


\subsection{Ingredients in the proof}
\label{ss:proofsketch}

Classical approaches to spectral multiplier theorems rely on Calder\'on--Zygmund theory, where ``the bad part'' is controlled, e.g., using Gaussian heat kernel bounds since these can be extended nicely to the right complex half-plane or H\"ormander-type conditions on the heat kernel. While in \cite{ZhaoZheng2023,Merz2022} complex-time heat kernel estimates for fractional Schr\"odinger operators have been obtained, they seem to be insufficient in the present situation: The bounds of \cite{ZhaoZheng2023} do not cover Hardy potentials, and also grow exponentially for large times. On the other hand, while the bounds in \cite{Merz2022} do hold for the complex-time heat kernel of the Hardy operator and do not grow for large times, they deteriorate too fast as the ratio $|\im z|/|\re z|$ grows.

The present approach instead uses the Helffer--Sj\"ostrand formula \cite[Proposition~7.2]{HelfferSjostrand1989}, which expresses functions of self-adjoint operators in terms of their resolvents. We will apply this formula to compactly supported functions of $(H+1)^{-1}$ (instead of to $F(H)$ directly) and aim for weighted estimates for such functions.
Note that the Helffer--Sj\"ostrand formula was already used in \cite[Section~3]{JensenNakamura1994} to prove $L^p$-boundedness of $F(H)$ for smooth, decaying symbols $F\in S^{-\epsilon}$, $\epsilon>0$, for Schr\"odinger operators under Kato-class assumptions on the potential, using polynomial $L^p$-resolvent bounds.
In \cite{Davies1995TF} an abstract version of this construction is developed for closed operators with real spectrum on Banach spaces under polynomial resolvent bounds.
In \cite[Section~6]{DuongRobinson1996}, H\"ormander and Davies--Helffer--Sj\"ostrand functional calculi are studied under Gaussian heat kernel assumptions. These assumptions are not satisfied by the fractional Schr\"odinger operators considered here.

To prove weighted estimates for the resolvent of $(H+1)^{-1}$ appearing in the Helffer--Sj\"ostrand, we use a quantitative operator-valued version of Jaffard's inverse-closedness theorem \cite{Jaffard1990} on $\Z^d$. The estimate preserves the offdiagonal decay exponent and provides polynomial control as the spectral parameter approaches the real axis. Applying the Helffer--Sj\"ostrand formula then gives bounds for compactly supported, sufficiently smooth functions of $(H+1)^{-1}$. Together with real-time heat kernel estimates, these bounds yield the localized multiplier estimates needed for the H\"ormander multiplier theorem.
In this regard, note that quantitative inverse estimates for matrices with off-diagonal decay have also been applied to spectral theory and smooth functional calculi, e.g., in \cite[\S~4]{Baskakov1997E}.

\smallskip
Let us give a more precise outline of the strategy now. The key is to control compactly supported, sufficiently smooth multipliers. In particular, using the Calder\'on--Zygmund theory of \cite{BlunckKunstmann2003}, it will be crucial to control $\|\one_{Q_a} m(H) \one_{Q_b}\|_{2\to q}$, where
$$
Q_a = a+[0,1)^d \quad\text{with}\quad a,b\in\Z^d.
$$
By an approximation argument it suffices to assume $m$ is $C^\infty$. We then write such a compactly supported, sufficiently smooth multiplier $m\in C_c^\infty(I)$, with a compact interval $I\subset(0,\infty)$, as
\begin{align}
  m(H) = \me{-H}\phi(K), \qquad K=(H+1)^{-1},
\end{align}
where
\begin{align}
  \phi(\lambda) =
  \begin{cases}
    \exp(\lambda^{-1}-1) m(\lambda^{-1}-1), & \quad \lambda\in(0,1], \\
    0, & \quad \lambda=0
  \end{cases}
\end{align}
satisfies $\phi\in C_c^\infty(J)$ for a fixed compact $J\subset(0,1)$ determined by $I$. Clearly, $m\mapsto\phi$ is smooth on the relevant compact sets and $\|\phi\|_{C^{L+1}}\lesssim_I \|m\|_{C^{L+1}}$ for any $L\in\N$. Using $L^2\to L^q$-operator norm bounds for projections of the heat kernel on two boxes, the key inequality, which allows us to control $\|\one_{Q_a} m(H) \one_{Q_b}\|_{2\to q}$ is
\begin{align}
  \label{eq:key-intro}
  \|\one_{Q_a}\phi(K)\one_{Q_b}\|_{2\to 2} \lesssim_\phi (1+|a-b|)^{-s}
\end{align}
for some $s>d$, where the implicit constant depends on finitely many derivatives of $\phi$. To show \eqref{eq:key-intro}, we use the Helffer--Sj\"ostrand formula\footnote{A detailed discussion of the Helffer--Sj\"ostrand formula and some of its extensions is contained, e.g., in \cite{DerezinskiGerard1997,Isozaki2023,IshidaMerz2026}.} to write
$$
\phi(K) = \frac{1}{2\pi i} \int_\C \bar\partial\widetilde\phi(z) (z-K)^{-1}\,dz\wedge d\bar z
$$
where $\bar\partial=\frac{\partial}{\partial\bar z}=\frac12(\partial_x+i\partial_y)$ and $dz\wedge d\bar z=-2idx\,dy$ for $z=x+iy$, and, for a fixed sufficiently large $L\in\N$, $\widetilde\phi\in C_c^\infty(\C)$ is an almost-analytic extension of $\phi$ of order $L$, satisfying $|\bar\partial\widetilde\phi(x+iy)| \lesssim_L |y|^L \|\phi\|_{C^{L+1}}$.
This almost-analytic extension allows us to control singularities of $(z-K)^{-1}$ on the real axis. The key ingredient is then an estimate on $\one_{Q_a}(z-K)^{-1}\one_{Q_b}$. We shall show that there is $M=M(d,\alpha)$ such that
\begin{align}
  \label{eq:key-intro2}
  \|\one_{Q_a}(u+iv-K)^{-1}\one_{Q_b}\|
  \lesssim (1+|a-b|)^{-s}|v|^{-M}, \quad u\in[-1,2], \quad 0<|v|\leq1.
\end{align}
This inversion estimate uses boundedness of $(1+|a-b|)^{s}\|\one_{Q_a}K\one_{Q_b}\|$, which in turn can be inferred from the sharp heat kernel bounds for $H$ by \cite{Bogdanetal2019,JakubowskiWang2020,Choetal2020}.
Estimate \eqref{eq:key-intro2} follows from an operator-valued version of Jaffard's lemma \cite{Jaffard1990} with polynomial control of the Jaffard norm of the inverse in terms of the original Jaffard norm and the operator norm of the inverse.
Quantitative estimates for the entries of inverses of operator-valued matrices are proved in \cite{Baskakov1997E}; see also \cite{BaskakovKrishtal2014} for a general framework for decay estimates of inverse operators.
A quantitative version of Jaffard's lemma for scalar-valued matrices with polynomial off-diagonal decay is established, e.g., in \cite{GrochenigKlotz2014,FangShin2020}. In particular, the coarse-localization and commutator method of \cite{FangShin2020} yields an explicit polynomial bound. Qualitative inverse-closedness results have been proved in the operator-valued setting in \cite[Corollary~4.5]{KohldorferBalazs2025W}, \cite[Theorem~II.12]{KohldorferBalazs2025}.
In Section~\ref{s:inverse}, we adapt the proof of \cite{FangShin2020} to show the needed operator-valued bound in Theorem~\ref{thm:QI}. Importantly, the bound preserves the decay exponent $s>d$ and has polynomial dependence on the operator norm of the inverse. Applying this estimate to $z-K$, we obtain \eqref{eq:key-intro2}, with a fixed power of $|\im z|^{-1}$. Thus, choosing the order of the almost-analytic extension sufficiently high yields \eqref{eq:key-intro}.
We expect that the present approach can also yield spectral multiplier theorems for larger classes of operators whose heat kernel decays only polynomially.


\subsection*{Organization and notation}

In Section~\ref{s:cz}, we reduce the proof of Theorem~\ref{thm:main} via Calder\'on--Zygmund theory, which we adapt from \cite{BlunckKunstmann2003}, to estimates for the ``bad part'', which we decompose into an ``easy'' and a ``difficult'' part; the easy part can be dealt with real-time heat kernel smoothing, while the difficult part is the one that in earlier works relied on suitable complex-time heat kernel estimates.
In Section~\ref{s:heat}, we recall the fundamental heat kernel bounds for $H$.
In Section~\ref{s:realtime}, we prepare heat kernel smoothing estimates which we then use to treat the easy part of the bad part in the Calder\'on--Zygmund decomposition.
In Section~\ref{s:baddifficult}, we treat the difficult part of the bad part of the Calder\'on--Zygmund decomposition by a dyadic decomposition of the multiplier.
The above-mentioned key inequalities \eqref{eq:key-intro}--\eqref{eq:key-intro2} that allow us to control every dyadic piece are proved in Section~\ref{s:hs-calculus}, which relies on a quantitative inversion theorem, which we state and prove in Section~\ref{s:inverse}.

We use the following notation throughout the manuscript.
For a set $\Omega$, we denote its complement by $\Omega^c$. The indicator function of $\Omega$ is denoted by $\one_\Omega$.
For $r>0$, $x\in\R^d$, and for $j\in\N_0$, let
\begin{align}
  \begin{split}
    B(x,r) & := \{z:|z-x|<r\}, \\
    A_j(x,r) & := B(x,2^jr)\setminus B(x,2^{j-1}r),\quad j\ge1, \qquad A_0(x,r):=B(x,r).
  \end{split}
\end{align}
We introduce further notation as we proceed.


\section{Calder\'on--Zygmund theory for $H$}
\label{s:cz}

In this section, we reduce the proof of Theorem~\ref{thm:main} to the two key inequalities~\eqref{eq:smoothmax-intro} and \eqref{eq:keydifficultpart} below. First, since $\|F\|_\infty<\infty$, the operator $F(H)$ is bounded on $L^2$ by the spectral theorem. Its $L^p$ boundedness for all $p_-<p<p_+$ follows once we prove the weak-type bound \eqref{eq:weak} at $p=p_0$ by Marcinkiewicz interpolation and duality. Thus, Theorem~\ref{thm:main} is proved, once we show
\begin{align}
  \label{eq:weaklp}
  |\{x:|F(H)f(x)|>\lambda \mathfrak N_s(F)\}|
  \lesssim \lambda^{-p_0}\|f\|_{p_0}^{p_0}, \quad \lambda>0.
\end{align}
Without loss of generality, we assume $\mathfrak N_s(F)>0$.
To show \eqref{eq:weaklp}, we first assume $f\in L^{p_0}\cap L^2$. Once \eqref{eq:weaklp} is shown for such functions, it extends to all $f\in L^{p_0}$ by density of $L^{p_0}\cap L^2$ in $L^{p_0}$.

With these preliminaries, we begin the proof of \eqref{eq:weaklp} by adapting the Calder\'on--Zygmund theory of \cite{BlunckKunstmann2003}, which is particularly suitable for operators whose heat kernels have singularities, such as $\me{-tH}$ with $\kappa>0$, see \eqref{eq:HK} below. 
By the usual dyadic Calder\'on--Zygmund decomposition (see, e.g., \cite[Ch.~I]{Stein1970}) applied to $|f|^{p_0}$ at height $\lambda^{p_0}$, we obtain pairwise disjoint maximal dyadic cubes $Q$ with
\begin{align}
  \label{eq:cubes}
  \sum_Q|Q| \leq \lambda^{-p_0}\|f\|_{p_0}^{p_0}, \qquad
  \lambda^{p_0} < \frac{1}{|Q|}\int_Q |f|^{p_0} \leq 2^d \lambda^{p_0}.
\end{align}
Define the cube average $f_Q:=|Q|^{-1}\int_Q f$, and ``the bad and good parts''
\begin{align}
  b_Q := (f-f_Q)\one_Q, \qquad
  g := f-\sum_Qb_Q.
\end{align}
Hence,
\begin{align}
  \label{eq:CZproperties}
  \begin{split}
    & f = g + \sum_Q b_Q, \\
    & \|g\|_{p_0} \leq C_d\|f\|_{p_0}, \qquad \|g\|_2^2 \leq C_d \lambda^{2-p_0}\|f\|_{p_0}^{p_0}, \\
    & \|b_Q\|_{p_0} \leq C_d \lambda |Q|^{1/p_0}.
  \end{split}
\end{align}
Note, by disjointness of the cubes,
\begin{align}
  \label{eq:disjoint}
  \begin{split}
    \sum_Q\one_Q & \leq 1 \quad \text{a.e.}, \\
    \left\|\sum_Q\one_Q\right\|_a & = \left(\sum_Q|Q|\right)^{1/a}, \quad 1\leq a<\infty.
  \end{split}
\end{align}
For each cube $Q$, let $\ell(Q)$ denote its side length and $z_Q$ its center. Given the degree of homogeneity $-\alpha$ of $H$, we further set
\begin{align}
  r_Q := 2\sqrt{d}\ell(Q), \qquad t_Q := r_Q^\alpha, \qquad U_Q := B(z_Q,8r_Q).
\end{align}
Note
$$
Q\subset B(z_Q,r_Q/2), \qquad
|B(z_Q,r_Q/2)| \sim_d |Q|, \qquad
\sum_Q|U_Q|\lesssim_d\sum_Q|Q|.
$$
Next, we adapt the bad part $b_Q$ to the operator $H$ and its semigroup. To this end, introduce an integer $N\in\N$ such that
\begin{align}
  \label{eq:Ncond}
  N > \frac{d}{\alpha q}
\end{align}
and
\begin{align}
  \begin{split}
    D_t := (1-\me{-tH})^N, \qquad
    D_Q := D_{t_Q}, \qquad
    E_Q := 1-D_Q.
  \end{split}
\end{align}
Since $\slim_{t\to0}\me{-tH}=\one_{L^2}$, $D_t$ acts as a refined approximation of zero and $E_Q$ as a refined approximation of the identity. Using $D_Q$ and $E_Q$, we decompose bad part further into
\begin{align}
  h_1:=\sum_Q E_Q b_Q, \qquad h_2:=\sum_QD_Q b_Q.
\end{align}
We call $h_1$ the {\em easy part} and $h_2$ the {\em difficult part}. We choose this terminology since $h_1$ can be controlled just using the real-time heat kernel bounds, while $h_2$ requires exploitation of the cancellations in $D_Q$ (with sufficiently high $N$). In the following two subsections, we first control the contribution from the easy and then from the difficult part. In the final subsection, we conclude the proof of \eqref{eq:weaklp}, under the assumption that the two key inequalities \eqref{eq:smoothmax-intro} and \eqref{eq:keydifficultpart} below hold.


\subsection{Analysis of the easy part}
\label{ss:easypart}
 
Following the Calder\'on--Zygmund theory of \cite{BlunckKunstmann2003}, let us introduce for $a\in[1,\infty]$,
\begin{align}
  \cn_{a,r}f(x) := r^{-d/a} \|\one_{B(x,r)}f\|_{L^a},
  \qquad
  \cm_a f(x) := \sup_{R>0}R^{-d/a} \|\one_{B(x,R)}f\|_{L^a}.
\end{align}
Recall $D_t = (1-\me{-tH})^N$.
In Proposition~\ref{prop:heatmax} below, we will show for $r=t^{1/\alpha}$ and $\phi\in L^2$,
\begin{align}
  \label{eq:smoothmax-intro}
  \cn_{q,r/2}\left((1-D_t)^*\phi\right)(z) \lesssim_N \cm_{p_0}\phi(x), \quad |x-z|<r/2.
\end{align}
Let us accept this bound for now. Since $r_Q=2\sqrt d\,\ell(Q)$ and $z_Q$ is the center of $Q$, we have $|x-z_Q|\leq r_Q/4$ for $x\in Q$. Hence, \eqref{eq:smoothmax-intro} is applicable with $z=z_Q$, $r=r_Q$, and every $x\in Q$. For an arbitrary $\phi\in L^2$, we thus obtain, by H\"older's inequality, \eqref{eq:CZproperties}, and \eqref{eq:smoothmax-intro} with $r=r_Q$, $z=z_Q$, and every $x\in Q$,
\begin{align*}
  |\langle \phi, E_Q b_Q\rangle|
  \leq \|b_Q\|_{p_0} \|\one_{Q} E_Q\phi\|_q
  \lesssim \lambda |Q|^{1/p_0} r_Q^{d/q} (\cm_{p_0}\phi)(x)
  \lesssim \lambda |Q|(\cm_{p_0}\phi)(x).  
\end{align*}
The left side is independent of $x\in Q$. Hence, using $|Q|^{-1}\int_Q=1$, we obtain
\begin{align*}
  |\langle \phi, E_Q b_Q\rangle|
  =\frac{1}{|Q|}\int_Q|\langle \phi, E_Q b_Q\rangle|\,dx
  \lesssim \lambda \int_Q (\cm_{p_0}\phi)(x)\,dx.
\end{align*}
Summing over the disjoint cubes and using Cauchy--Schwarz and \eqref{eq:disjoint} yields
\begin{align}
  \label{eq:h1dual}
  \begin{split}
    |\langle \phi,h_1\rangle|
    & \lesssim \lambda\int_{\R^d}\left(\sum_Q\one_Q(x)\right)(\cm_{p_0}\phi)(x)\,dx
    \lesssim \lambda \left(\sum_Q|Q|\right)^{1/2} \|\cm_{p_0}\phi\|_2 \\
    & \lesssim \lambda \left(\sum_Q|Q|\right)^{1/2} \|\phi\|_2.
  \end{split}
\end{align}
Here we used $p_0<2$ to use the $L^2$-boundedness of $\cm_{p_0}$ by the Hardy--Littlewood maximal inequality. Taking the supremum over all $\phi\in L^2$ thus shows
\begin{align}
  \label{eq:h1norm}
  \|h_1\|_2 \lesssim \lambda\left(\sum_Q|Q|\right)^{1/2}.
\end{align}


\subsection{Analysis of the difficult part}
\label{ss:difficultpart}

For each $Q$, set
$$
G_Q := \one_{U_Q^c} F(H) D_Q\one_Q.
$$
As in the previous subsection, we argue by duality. Thus, let $\phi\in L^2\cap L^q$ be arbitrary. In Theorem~\ref{thm:high-regularity-cancellation} below, we show for $\phi\in L^2$, $r=t^{1/\alpha}$, and $x,y\in\R^d$ with $|x-y|<r/2$,
\begin{align}
  \label{eq:keydifficultpart}
  \cn_{q,r/2}\left( \left[F(H)D_t\right]^*\one_{B(y,4r)^c} \phi \right)(y)
  \lesssim_{d,\alpha,\kappa,p_0,s,N,\psi,\chi}\fh_s(F)\,(\cm_2 \phi)(x),
\end{align}
uniformly in $Q,x,y,\phi,F$.
Accepting this bound for now, we have, for $t=t_Q$, $r=r_Q$, $y=z_Q$, and every $x\in Q$,
\begin{align*}
  |\langle \phi,G_Qb_Q\rangle|
  & \leq \|b_Q\|_{p_0} \|\one_Q(F(H)D_Q)^*\one_{U_Q^c}\phi\|_q \\
  & \lesssim \lambda |Q|^{1/p_0} r_Q^{d/q} \fh_s(F) (\mathcal M_2\phi)(x) 
    \lesssim \lambda |Q| \fh_s(F) (\mathcal M_2\phi)(x)
\end{align*}
Here we applied \eqref{eq:keydifficultpart} to $\one_{U_Q^c}\phi$, using
$$
\one_{B(z_Q,4r_Q)^c}\one_{U_Q^c}\phi=\one_{U_Q^c}\phi,
\qquad
\cm_2(\one_{U_Q^c}\phi)\leq\cm_2\phi,
$$
and $Q\subset B(z_Q,r_Q/2)$.
As the left-hand side of the above estimate is independent of $x$, we use $|Q|^{-1}\int_Q dx=1$ to infer
\begin{align*}
  |\langle \phi,G_Qb_Q\rangle|
  = \frac{1}{|Q|}\int_Q |\langle \phi,G_Qb_Q\rangle|\,dx
  \lesssim \lambda \fh_s(F) \int_Q(\mathcal M_2\phi)(x)\,dx.
\end{align*}
Summing over the disjoint cubes yields
\begin{align*}
  |\langle \phi,\sum_Q G_Qb_Q\rangle|
  & \lesssim \lambda \fh_s(F) \int\left(\sum_Q\one_Q\right)(\mathcal M_2\phi)(x)\,dx \\
  & \lesssim \lambda \fh_s(F) \left(\sum_Q|Q|\right)^{1/p_0} \|\mathcal M_2\phi\|_q 
    \lesssim \lambda \fh_s(F) \left(\sum_Q|Q|\right)^{1/p_0} \|\phi\|_q
\end{align*}
where we again used \eqref{eq:disjoint} and $q>2$ to apply the Hardy--Littlewood maximal inequality. Thus, by density of $L^2\cap L^q$ in $L^q$, followed by duality,
\begin{align}
  \label{eq:extnorm}
  \left\|\sum_Q G_Q b_Q\right\|_{p_0}
  \lesssim \lambda \fh_s(F) \left(\sum_Q|Q|\right)^{1/p_0}.
\end{align}


\subsection{Conclusion}

By Chebyshev's inequality, \eqref{eq:CZproperties}, \eqref{eq:h1norm}, and \eqref{eq:extnorm}, we obtain, with the enlarged exceptional set $U=\bigcup_Q U_Q$,
\begin{align}
  \label{eq:exc}
  \begin{split}
    |\{|F(H)g|>\Lambda/3\}| & \lesssim \Lambda^{-2} \mathfrak N_s(F)^2 \lambda^{2-p_0}\|f\|_{p_0}^{p_0}, \\
    |\{|F(H)h_1|>\Lambda/3\}| & \lesssim \Lambda^{-2} \mathfrak N_s(F)^2 \lambda^2\sum_Q|Q| \lesssim \Lambda^{-2} \mathfrak N_s(F)^2\lambda^{2-p_0}\|f\|_{p_0}^{p_0}, \\
    |\{x\in U^c:|F(H)h_2(x)|>\Lambda/3\}| & \lesssim \Lambda^{-p_0} \lambda^{p_0} \fh_s(F)^{p_0} \sum_Q|Q| \lesssim \Lambda^{-p_0} \fh_s(F)^{p_0} \|f\|_{p_0}^{p_0}.
  \end{split}
\end{align}
Using $|U|\lesssim \sum_Q|U_Q|\lesssim \lambda^{-p_0}\|f\|_{p_0}^{p_0}$ and taking $\Lambda=\lambda\mathfrak N_s(F)$ yields~\eqref{eq:weaklp} and concludes the proof of Theorem~\ref{thm:main}, assuming \eqref{eq:smoothmax-intro} (Proposition~\ref{prop:heatmax}) and \eqref{eq:keydifficultpart} (Theorem~\ref{thm:high-regularity-cancellation}).


\section{Heat kernel of $H$}
\label{s:heat}

We recall from \cite{Bogdanetal2019,JakubowskiWang2020,Choetal2020} the sharp heat kernel bounds for $H$. For $d\in\N$, $\alpha\in(0,\min\{2,d\})$, $\kappa\leq\kappa_{\rm c}$, $t>0$, $r:=t^{1/\alpha}$, and $x,y\in\R^d\setminus\{0\}$, they state
\begin{align}
  \label{eq:HK}
  \begin{split}
    \me{-tH}(x,y) & \sim \me{-t(-\Delta)^{\alpha/2}}(x,y) w_r(x)w_r(y), \\
    \me{-t(-\Delta)^{\alpha/2}}(x,y) & \sim \frac{t}{(r+|x-y|)^{d+\alpha}}, \\
    w_r(x) & := \min\left\{1, \frac{|x|}{r}\right\}^{-\delta}.
  \end{split}
\end{align}
We record a simple weighted estimate involving $w_r$ for $\kappa>0$.

\begin{lemma}
  \label{lem:cubew}
  Let $1\leq u<\infty$ and $r>0$. Let $0<\kappa\leq\kappa_{\rm c}$ and $\delta$ be defined as in \eqref{eq:defdelta}. Then, the condition $u\delta<d$ is necessary and sufficient for
  \begin{align*}
    \sup_{a\in \Z^d}\|\one_{Q_a}w_r\|_{L^u} < \infty.
  \end{align*}
  When it holds, then, uniformly in $r>0$,
  \begin{align}
    \label{eq:weight}
    c_{d,u,\delta}(1+r^\delta)
    \leq \sup_{a\in\Z^d}\|\one_{Q_a} w_r\|_{L^u}
    \leq C_{d,u,\delta}(1+r^\delta).
  \end{align}
\end{lemma}

\begin{proof}
  The upper bound follows from
  $$
  \sup_{a\in\Z^d}\int_{Q_a}\left(1+r^{u\delta}|x|^{-u\delta}\right)\,dx
  \lesssim 1 + r^{\delta u}, \quad u\delta<d,
  $$
  uniformly in $r>0$.
  For the lower bound, $w_r\geq1$ and, on $Q_0$, $w_r(x)\geq r^\delta|x|^{-\delta}$, whose $L^u(Q_0)$ norm is a positive finite multiple of $r^\delta$. The integral over $Q_0$ diverges when $u\delta\geq d$.
\end{proof}

We now prove localized smoothing estimates for the heat kernel and the resolvent.

\begin{proposition}
  \label{prop:cubes}
  For every $a,b\in\Z^d$, we have
  \begin{align}
    \label{eq:Ecube}
    \|\one_{Q_a} \me{-H}\one_{Q_b}\|_{2\to q} & \lesssim_q (1+|a-b|)^{-d-\alpha}, \\
    \label{eq:Kcube}
    \|\one_{Q_a}(1+H)^{-1}\one_{Q_b}\|_{2\to2} & \lesssim (1+|a-b|)^{-d-\alpha}.
  \end{align}
  In particular, $\me{-H}:L^2(\R^d)\to L^q(\R^d)$ is bounded.
\end{proposition}

\begin{proof}
  We first show \eqref{eq:Ecube}.
  For all $x\in Q_a$ and $y\in Q_b$, we have, by the triangle inequality,
  \begin{align}
    \label{eq:cubedistance}
    1+|x-y| \gtrsim_d 1+|a-b|.
  \end{align}
  Hence, by \eqref{eq:HK},
  $$
  |(\one_{Q_a}\me{-H}\one_{Q_b}f)(x)|
  \lesssim (1+|a-b|)^{-d-\alpha} \one_{Q_a} w_1(x) \|\one_{Q_b}w_1\|_2 \|f\|_2.
  $$
  If $\kappa<0$, we estimate this further from above by $(1+|a-b|)^{-d-\alpha} \one_{Q_a}(x) \|f\|_2$. Thus, taking the $L^q(Q_a)$ norm and using Lemma~\ref{lem:cubew} (when $\kappa>0$) yields~\eqref{eq:Ecube}.

  Let us now show the global $L^2\to L^q$-boundedness of $\me{-H}$. Writing $f=\sum_b f\one_{Q_b}$, we have, by \eqref{eq:Ecube},
  \begin{align*}
    \|\one_{Q_a}\me{-H}f\|_q
    \lesssim \sum_b(1+|a-b|)^{-d-\alpha}\|\one_{Q_b}f\|_2.
  \end{align*}
  Thus, by Young's inequality $\|F\ast G\|_q\leq \|F\|_1\|G\|_q$ and $\ell^2\subset\ell^q$ for $q\geq2$,
  \begin{align*}
    \|\me{-H}f\|_q \lesssim \|f\|_2. 
  \end{align*}
  
  \smallskip
  We finally show \eqref{eq:Kcube} by writing $(1+H)^{-1}=\int_0^\infty\me{-t(1+H)}\,dt$.
  Consider first those $a,b$ for which $|a-b|>4\sqrt d+4$. Then, by \eqref{eq:HK} and \eqref{eq:cubedistance},
  \begin{align*}
    \|\one_{Q_a}\me{-tH}\one_{Q_b}\|_{2\to2}^2
    \leq \int_{Q_a}dx \int_{Q_b}dy\, |\me{-tH}(x,y)|^2
    \lesssim t^2\, (1+|a-b|)^{-2(d+\alpha)} \cdot (1+t^{4\delta_+/\alpha}).
  \end{align*}
  Since $\int_0^\infty \me{-t}t(1+t^{2\delta_+/\alpha})\,dt<\infty$, this yields \eqref{eq:Kcube} for those separated cubes. For $|a-b|\lesssim1$, \eqref{eq:Kcube} follows from $\|(1+H)^{-1}\|_{2\to2}\leq1$.
\end{proof}


\section{Maximal estimates for the easy part in the Calder\'on--Zygmund decomposition}
\label{s:realtime}

Recall $N\in\N$ is such that \eqref{eq:Ncond}, i.e., $N>d/(\alpha q)$ holds.
The goal of this section is to prove the following proposition, which we used to control the easy part of the bad part in the Calder\'on--Zygmund decomposition (Subsection~\ref{ss:easypart}).

\begin{proposition}
  \label{prop:heatmax}
  For $t>0$, $r=t^{1/\alpha}$, fixed $L\in\N$ and $\ell\in\{1,...,L\}$, and $\phi\in L^2$,
  \begin{align}
    \label{eq:heatmax}
    \cn_{q,r/2}(\me{-\ell t H}\phi)(z) \lesssim_L (\cm_{p_0}\phi)(x), \quad |x-z|<r/2.
  \end{align}
  In particular, with $N>d/(\alpha q)$ and $D_t = (1-\me{-tH})^N$,
  \begin{align}
    \label{eq:smoothmax}
    \cn_{q,r/2}\left((1-D_t)^*\phi\right)(z) \lesssim_N (\cm_{p_0}\phi)(x), \quad |x-z|<r/2.
  \end{align}
\end{proposition}

We first record a simple weighted estimate, whose proof we omit.

\begin{lemma}
  \label{lem:weight}
  For $0\leq b<d$ and $R\geq r>0$, we have
  \begin{align*}
    \sup_{z\in\R^d}\int_{B(z,R)}\left(1+\frac{r}{|x|}\right)^b\,dx
    \lesssim_{b,d} R^d.
  \end{align*}
  In particular, for $q\max\{\delta,0\}<d$, we have
  \begin{align*}
    \|\one_{B(z,R)}w_r\|_q \lesssim R^{d/q}
  \end{align*}
  uniformly for $R\geq r$.
\end{lemma}

Next, we prove another smoothing offdiagonal estimate for the heat kernel of $H$.

\begin{lemma}
  \label{lem:heatloc}
  Fix $L\in\N$. Then, uniformly in $1 \leq \ell \leq L$, $t>0$, $r=t^{1/\alpha}$, $z\in\R^d$, and $j\in\N_0$,
  \begin{align}
    \label{eq:heatloc}
    \|\one_{B(z,r/2)}\me{-\ell t H}\one_{A_j(z,r)}\|_{p_0 \to q}
    \lesssim_L r^{-d(1/p_0-1/q)} 2^{-j(d+\alpha-d/q)}.
  \end{align}
\end{lemma}

\begin{proof}
  For $u\in B(z,r/2)$ and $v\in A_j(z,r)$, the free heat kernel factor in \eqref{eq:HK} is $\lesssim r^{-d} 2^{-j(d+\alpha)}$, with a fixed adjustment for $j=0,1,2$. Hence, by H\"older's inequality,
  \begin{align*}
    \|\one_{B(z,r/2)} \me{-\ell t H}\one_{A_j(z,r)}\|_{p_0\to q}
    & \lesssim_L r^{-d} 2^{-j(d+\alpha)} \|\one_{B(z,r/2)}w_r\|_q \|\one_{A_j(z,r)}w_r\|_{p_0'} \\
    & \lesssim_L r^{-d} 2^{-j(d+\alpha)} \cdot r^{d/q}(2^jr)^{d/q}.
  \end{align*}
  Using $p_0'=q$ and $2d/q-d=-d(1/p_0-1/q)$ concludes the proof.
\end{proof}

We are now ready to give the

\begin{proof}[Proof of Proposition~\ref{prop:heatmax}]
  We decompose $\phi=\sum_{j\geq0}\one_{A_j(z,r)}\phi$. Since $|x-z|<r/2$ implies that the $j$-th summand has support contained in $B(x,C2^jr)$ for some $C>0$, we have
  $$
  \|\one_{A_j(z,r)}\phi\|_{p_0} \lesssim (2^jr)^{d/p_0} (\cm_{p_0}\phi)(x).
  $$
  By Lemma~\ref{lem:heatloc},
  \begin{align*}
    \|\one_{B(z,r/2)}\me{-\ell tH}\phi\|_q
    & \lesssim_L r^{-d(1/p_0-1/q)}r^{d/p_0}\sum_{j\geq0}2^{-j(d+\alpha-d/q-d/p_0)}(\cm_{p_0}\phi)(x) \\
    & \lesssim_L r^{d/q} \sum_{j\geq0}2^{-j\alpha}(\cm_{p_0}\phi)(x),
  \end{align*}
  which is \eqref{eq:heatmax} after dividing both sides by $r^{d/q}$. Using
  $$
  1-(1-\me{-t H})^N=\sum_{\ell=1}^N(-1)^{\ell+1}\binom{N}{\ell}\me{-\ell tH},
  $$
  the triangle inequality, and \eqref{eq:heatmax}, shows also \eqref{eq:smoothmax}.
\end{proof}


\section{The difficult part in the Calder\'on--Zygmund decomposition}
\label{s:baddifficult}

The main result of this section is the following theorem, which allows to treat the difficult part in the Calder\'on--Zygmund decomposition. It is the main novel technical contribution in this paper and it seems to be difficult (if not impossible) to prove with presently available complex-time heat kernel bounds. Recall $D_t = (1-\me{-tH})^N$ and $S_*$ is defined in \eqref{eq:indices}.

\begin{theorem}
  \label{thm:high-regularity-cancellation}
  Assume $d\in\N$, $\alpha\in(0,\min\{2,d\})$, $\kappa\leq\kappa_{\rm c}$, $s>S_*$, $p_-<p_0<2$, $q=p_0'$, and $N\in\N$ is such that \eqref{eq:Ncond} holds, i.e., $N>d/(\alpha q)$. If $F:[0,\infty)\to\C$ is measurable with $\fh_s(F)<\infty$, then for every $t>0$, $r=t^{1/\alpha}$, $x,y\in\R^d$ with $|x-y|<r/2$ and $f\in L^2(\R^d)$,
  \begin{align}
    \label{eq:high-regularity-cancellation}
    \cn_{q,r/2}\left(\left[F(H)D_t\right]^* \one_{B(y,4r)^c} f \right)(y)
    \lesssim_{d,\alpha,\kappa,p_0,s,N,\psi,\chi} \fh_s(F)\,(\cm_2 f)(x).
  \end{align}
  The implicit constant is independent of $t,x,y,f,F$.
  The same estimate is valid for $F(H)D_t$ instead of $\left[F(H)D_t\right]^*$.
\end{theorem}

The proof idea is classical and follows \cite{BlunckKunstmann2003}.
The starting point is a dyadic decomposition of $F(H)D_t$ involving compactly supported pieces
$$
m_{k,t}(2^{-k}H)=\psi(2^{-k}H)F(H)(1-\me{-tH})^N.
$$
In Lemma~\ref{lem:product} below, we show that the dyadic pieces have the required regularity, i.e.,
\begin{align}
  \label{eq:product-intro}
  \|m_{k,t}\|_{H^s} \lesssim \fh_s(F)\min\{1,(t2^k)^N\}.
\end{align}
Let us then consider $m\in H^s(\R)$ which is supported in some compact interval $I\subset(0,\infty)$. The central ingredient to show \eqref{eq:high-regularity-cancellation} is the estimate
\begin{align}
  \label{eq:window-intro}
  \|\one_{B(z,\rho)}m(\rho^\alpha H)\one_{A_j(z,\rho)}\|_{2\to q} \lesssim_{I,s,q} \rho^{-d(1/2-1/q)} 2^{-j(d/2+\alpha)}\|m\|_{H^s}.
\end{align}
We will prove it in Theorem~\ref{thm:window}; for $\alpha=2$, one can use complex-time heat kernel estimates to prove it, but this seems to be impossible for $\alpha<2$. With \eqref{eq:window-intro} at hand, it is a classical matter to show for $|x-y|<r/2$ and $f\in L^2$,
\begin{align}
  \label{eq:single-intro}
  \cn_{q,r/2}\left(m(\sigma H)^*\one_{B(y,4r)^c}f\right)(y)
  \lesssim \|m\|_{H^s} (\cm_2f)(x)
  \begin{cases}
    (\sigma/t), & \quad 0 < \sigma \leq t, \\
    (\sigma/t)^{d/(\alpha q)}, & \quad \sigma \geq t,
  \end{cases}
\end{align}
which is the single-scale version of the claimed bound \eqref{eq:high-regularity-cancellation}. We prove it in Proposition~\ref{prop:single}.
Thus, assuming \eqref{eq:window-intro}, \eqref{eq:single-intro}, and \eqref{eq:product-intro}, let us now give the

\begin{proof}[Proof of Theorem~\ref{thm:high-regularity-cancellation}]
  We use the dyadic partition of unity $\sum_{k\in\Z}\psi(2^{-k}\lambda)=1$ for all $\lambda>0$ to decompose $F(H)D_t$ into dyadic pieces. To this end, define
  \begin{align}
    m_{k,t}(\lambda) := \psi(\lambda) F(2^k \lambda) \left(1-\exp\left(-t2^k \lambda\right)\right)^N, \quad \lambda>0.
  \end{align}
  Since $m_{k,t}(2^{-k}H)=\psi(2^{-k}H) F(H)\left(1-\me{-tH}\right)^N$, the spectral theorem in $L^2$ yields
  \begin{align}
    F(H)D_t = \sum_{k\in \Z}m_{k,t}(2^{-k}H).
  \end{align}
  By Lemma~\ref{lem:product} (with $u=t2^k$ there),
  \begin{align}
    \label{eq:product-intro-proof}
    \|m_{k,t}\|_{H^s} \lesssim \fh_s(F) \min\{1,t2^k\}^N, \quad k\in\Z,\,t>0.
  \end{align}
  By absolute continuity of the spectral measure of $H$, we may replace each $m_{k,t}$ by its continuous representative without changing $m_{k,t}(2^{-k}H)$.
  Thus, by Proposition~\ref{prop:single} with $\sigma=2^{-k}$ applied to $m_{k,t}(2^{-k}H)^*$,
  \begin{align}
    \label{eq:ksum}
    \begin{split}
      & \cn_{q,r/2}\left(m_{k,t}(2^{-k}H)^*\one_{B(y,4r)^c}f\right)(y) \\
      & \quad \leq C \fh_s(F)(\cm_2f)(x) \min\left\{1,(t2^k)^N\right\}
        \begin{cases}
          (t2^k)^{-d/(\alpha q)}, & \quad t2^k \leq 1, \\
          (t2^k)^{-1}, & \quad t2^k \geq 1.
        \end{cases}
    \end{split}
  \end{align}
  Since $N>d/(\alpha q)$, both tails are summable in $k$ with
  \begin{align}
    \label{eq:dyadsum}
    \sup_{t>0}\left(\sum_{k:\, t2^k\leq1}(t2^k)^{N-d/(\alpha q)} + \sum_{k:\,t2^k>1}(t2^k)^{-1}\right)<\infty.
  \end{align}
  Thus, the triangle inequality and \eqref{eq:ksum}--\eqref{eq:dyadsum} conclude the proof of \eqref{eq:high-regularity-cancellation}.
\end{proof}

We now prove \eqref{eq:product-intro-proof}.

\begin{lemma}
  \label{lem:product}
  Let $I\subset(0,\infty)$ be a compact interval, possibly containing $\supp\chi$ in its interior, $s\geq0$, and $N\in\N$. With $E_u=(1-\me{-uv})^N$, we have for every $u>0$ and $G$ with $\chi G \in H^s$,
  \begin{align}
    \label{eq:product}
    \|\psi(\cdot)G(\cdot)E_u(\cdot)\|_{H^s} \lesssim_{s,N,\psi,\chi} \min\{1,u^N\}\|\chi(\cdot)G(\cdot)\|_{H^s}
  \end{align}
  with the implicit constant being independent of $u$.
\end{lemma}

\begin{proof}
  Since $\psi\chi=\psi$, we can write $\psi G E_u=(\psi E_u)(\chi G)$. Thus, it remains to show
  $$
  \|(\psi E_u)f\|_{H^s}\lesssim\min\{1,u^N\}\|f\|_{H^s}, \quad f\in H^s.
  $$
  Recall that for any fixed compact $K\subset\R$, $s\geq0$, and any integer $k>s+1$, by Sobolev embedding,
  $$
  \|hf\|_{H^s}\lesssim_{s,k,K}\|h\|_{C^k}\|f\|_{H^s}, \quad h\in C_c^k(K).
  $$
  Hence, it suffices to show
  \begin{align}
    \label{eq:product-aux}
    \|E_u\|_{C^k(\supp\chi)}\lesssim_k \min\{1,u^N\}.
  \end{align}
  For $0< u\leq1$, on $\supp\psi$, we write
  \begin{align*}
    E_u(v) = (uv)^N \left(\int_0^1 \me{-\tau uv}\,d\tau\right)^N.
  \end{align*}
  All $v$-derivatives of $u^{-N}E_u(v)$ of any fixed finite order are uniformly bounded on $\supp\chi$. For $u\geq1$, we expand
  \begin{align*}
    E_u(v) = \sum_{\ell=0}^N(-1)^\ell\binom{N}{\ell}\me{-\ell uv}.
  \end{align*}
  The $\ell=0$-term is constant. On the other hand, for every $k\in\N_0$ and $v\geq a>0$, the quantity $u^k\me{-\ell uv}$ (arising from the $k$-th $v$-derivative of the $\ell$-th summand) is uniformly bounded for $\ell\in\{1,...,N\}$. This shows \eqref{eq:product-aux} and concludes the proof.
\end{proof}

We now prove \eqref{eq:single-intro} using \eqref{eq:window-intro}, which is proved in Theorem~\ref{thm:window} below.

\begin{proposition}
  \label{prop:single}
  Let $s>S_*$, $p_-<p_0<2$, and $q=p_0'$.
  Let $m\in H^s(\R)$ be supported in a fixed compact interval $I\subset(0,\infty)$. Let $\sigma,t>0$, and $r=t^{1/\alpha}$. Then, for $|x-y|<r/2$ and $f\in L^2(\R^d)$,
  \begin{align}
    \label{eq:single}
    \cn_{q,r/2}\left(m(\sigma H)^*\one_{B(y,4r)^c}f\right)(y)
    \lesssim \|m\|_{H^s}(\cm_2f)(x)
    \begin{cases}
      (\sigma/t), & \quad 0<\sigma\leq t, \\
      (\sigma/t)^{d/(\alpha q)}, & \quad \sigma \geq t.
    \end{cases}
  \end{align}
  The same estimate also holds with $m(\sigma H)^*$ replaced with $m(\sigma H)$.
\end{proposition}

\begin{proof}
  Let $\rho=\sigma^{1/\alpha}$ and $f_y=\one_{B(y,4r)^c}f$. We distinguish between the cases $\rho \geq r$ (i.e., $\sigma\geq t$) and $0<\rho\leq r$ (i.e., $\sigma\leq t$).

  \medskip
  \noindent
  {\em Case $\rho\geq r$:}
  Here $B(y,r/2)\subset B(y,\rho)$.
  We decompose $f_y=\sum_{j\geq0}f_y\one_{A_j(y,\rho)}$.
  For $j\in\N_0$ and $|x-y|<r/2\leq\rho/2$, there is $C>0$ such that
  \begin{align*}
    A_j(y,\rho)\subset B(x,C2^j\rho), \qquad
    \|\one_{A_j(y,\rho)}f_y\|_2 \leq C(2^j\rho)^{d/2} (\mathcal M_2f)(x).
  \end{align*}
  We now apply Theorem~\ref{thm:window} (with $z=y$ there) to $\overline m$ and sum using the triangle inequality to obtain
  \begin{align*}
    \|\one_{B(y,r/2)}m(\sigma H)^*f_y\|_q
    & \lesssim \|m\|_{H^s} \rho^{-d(1/2-1/q)}\sum_{j\geq0}2^{-j(d/2+\alpha)} (2^j\rho)^{d/2} (\mathcal M_2 f)(x) \\
    & \lesssim \|m\|_{H^s} \rho^{d/q} (\mathcal M_2f)(x) \sum_{j\geq0}2^{-\alpha j}
      \lesssim \|m\|_{H^s} \rho^{d/q} (\mathcal M_2f)(x).
  \end{align*}
  Dividing both sides by $(r/2)^{d/q}$ proves the claimed bound.
  
  \medskip
  \noindent
  {\em Case $0<\rho\leq r$:}
  We cover $B(y,r/2)$ with $K\leq C_d(r/\rho)^d$ many balls $B(z_\ell,\rho)$ centered at $z_\ell\in B(y,r/2)$. Thus,
  $$
  \one_{B(y,r/2)} \leq \sum_{\ell=1}^K\one_{B(z_\ell,\rho)}.
  $$
  We decompose $f_y=\sum_{j\geq0}f_y\one_{A_j(z_\ell,\rho)}$.
  Note
  $$
  |v-z_\ell| \geq |v-y| - |y-z_\ell| > \frac{7r}{2}, \quad v\in\supp f_y.
  $$
  Hence, for $(\one_{A_j(z_\ell,\rho)}f_y)(v)$ to be nonzero, we necessarily have $2^j\rho \geq 3r$ (here $j\geq1$). On the other hand, if $2^j\rho \geq 3r$, then we also have $A_j(z_\ell,\rho)\subset B(x,C2^j\rho)$ for some $C>0$ since $|z_\ell-x|\leq|z_\ell-y|+|x-y|\leq r$. Hence,
  $$
  \|\one_{A_j(z_\ell,\rho)}f_y\|_2 \leq C (2^j\rho)^{d/2} (\mathcal M_2f)(x), \quad \text{if} \ 2^j\rho \geq3r.
  $$
  Thus, by Theorem~\ref{thm:window} (with $z=z_\ell$ there),
  \begin{align*}
    \|\one_{B(z_\ell,\rho)}m(\sigma H)^* f_y\|_q
    & \lesssim \|m\|_{H^s} \rho^{d/q}(\mathcal M_2f)(x) \sum_{2^j\rho\geq3r}2^{-j\alpha} \\
    & \lesssim \|m\|_{H^s} \rho^{d/q} (\rho/r)^\alpha (\mathcal M_2 f)(x).
  \end{align*}
  Now, using the covering $\one_{B(y,r/2)} \leq \sum_{\ell=1}^K\one_{B(z_\ell,\rho)}$, summing the $q$-th powers over the covering balls (which may overlap) yields,
  \begin{align*}
    \|\one_{B(y,r/2)}m(\sigma H)^* f_y\|_q
    & \lesssim K^{1/q} \|m\|_{H^s} \rho^{d/q} (\rho/r)^\alpha (\mathcal M_2f)(x) \\
    & \lesssim \|m\|_{H^s}r^{d/q}(\rho/r)^\alpha (\mathcal M_2f)(x).
  \end{align*}
  where we used $\rho\leq r$. Dividing both sides by $(r/2)^{d/q}$ yields the claimed bound.
  
  \smallskip
  Clearly, the above arguments also apply to $m(\sigma H)$ in place of $m(\sigma H)^*$.
\end{proof}


\section{Quantitative inversion of operator-valued matrices}
\label{s:inverse}

As explained in Subsection~\ref{ss:proofsketch}, it will be important to have quantitative control of $\one_{Q_a}(z-(1+H)^{-1})^{-1}\one_{Q_b}$ with respect to the distance $|a-b|$ between the cube centers $a,b\in\Z^d$. Quantitative control of the norm of the inverse of an invertible element in a given Banach (sub)algebra by its ``smoothness'' and ``spectral data'' is a classical topic in functional analysis, going back at least to \cite{Nikolski1999}\footnote{The term ``smoothness'' goes back to the study of norm-controlled inversion in $C(\mathbb{T})$ via Fourier series \cite{Nikolski1999}; see also the introduction in \cite{GrochenigKlotz2014}.}. A Banach subalgebra $\mathcal A$ of a Banach algebra $\mathcal B$ (with a common identity) is called inverse-closed, if $A\in\mathcal A$ has an inverse $A^{-1} \in \mathcal B$ implies that one also has $A^{-1}\in\mathcal A$. For a Banach subalgebra $\mathcal A$ which is inverse-closed in $\mathcal B$, one says $\mathcal A$ admits a norm-controlled inversion in $\mathcal B$ if there is a function $h:\R_+\times\R_+\to\R_+$ such that
\begin{align*}
  \|A^{-1}\|_{\mathcal A} \leq h(\|A^{-1}\|_{\mathcal B},\|A\|_{\mathcal A}) 
\end{align*}
for all $A\in\mathcal A$ with inverse $A^{-1}\in\mathcal B$. One then calls $\|A^{-1}\|_{\mathcal B}$ the ``spectral data'', while $\|A\|_{\mathcal A}$ measures ``smoothness'' of $A$.

\smallskip
The relevant Banach algebra in our case is $\cb(\cx)$, i.e., the algebra of bounded linear operators over $\cx=\ell^2(\Z^d;\ch)$ for a given Hilbert space $\ch$. (In our case, $\ch$ is $L^2([0,1)^d)$ by the identification of $L^2(\R^d)$ with $\ell^2(\Z^d;L^2([0,1)^d))$, but we do not assume this for now.) The ``smoothness'' then is measured in our case using the subalgebra $\J_s:=\J_s(\Z^d;\cb(\ch))$ equipped with the norm
\begin{align}
  \|A\|_{\J_s} := \sup_{a,b\in\Z^d}(1+|a-b|)^s\|A_{ab}\|_{\cb(\ch)}, \quad s>d.
\end{align}
Since $\sum_{n\in\Z^d}(1+|n|)^{-s}<\infty$, $\J_s$ is a unital Banach $*$-algebra when equipped with an equivalent submultiplicative norm, and a subalgebra of $\cb(\cx)$.
Thus, ``smoothness'' here means polynomial offdiagonal decay of $(A_{ab})_{ab}$, i.e., a measure of ``locality''; the higher $s$ is, the more local elements in $\J_s$ are. Thinking of the heat kernel $\me{-H}$ and \eqref{eq:HK}, one sees that a discretization of $\me{-H}$, i.e., $\one_{Q_a}\me{-H}\one_{Q_b}$ belongs to $\J_s$ for all $d<s\leq d+\alpha$ when we take $\ch=L^2([0,1)^d)$.

Quantitative inverse-closedness estimates for operator-valued matrices are proved, e.g., in \cite[Theorem~6]{Baskakov1997E}, where polynomial off-diagonal decay is measured using a one-dimensional index set. However, for $d>1$, an enumeration of $\Z^d$ does not in general preserve the distance estimates required in our situation to apply that theorem directly.
More recently, \cite{BaskakovKrishtal2014} develop a general theory of decay estimates for inverse operators, including quantitative bounds in Wiener classes \cite[Theorem~7.2]{BaskakovKrishtal2014} and accommodate higher dimensions. Our application requires bounds in $\J_s$, $s>d$, which preserve the decay exponent $s$ and depend polynomially on the inverse operator norm. Such bounds are not directly available in \cite[Theorem~7.2]{BaskakovKrishtal2014}.
Let us also mention the more recent results \cite[Corollary~4.5]{KohldorferBalazs2025W} and \cite[Theorem~II.12]{KohldorferBalazs2025} which prove qualitative inverse-closedness for the operator-valued Jaffard algebras considered here. In the scalar case $\ch=\C$, quantitative inverse-closedness with polynomial norm control is studied in \cite{GrochenigKlotz2014,FangShin2020}.
For instance, \cite[Theorem~2]{FangShin2020} shows polynomial norm-controlled inversion for scalar matrices over relatively separated subsets of $\R^d$. Taking $p=\infty$, $q=2$, and $r=s>d$ there, then \cite[Theorem~2]{FangShin2020} asserts for $s\neq d+1$,
$$
\|A^{-1}\|_{\J_s} \leq C_{d,s} \|A^{-1}\|\left(\|A\|_{\J_s} \|A^{-1}\|\right)^{s/\min\{1,s-d\}}.
$$
For $s=d+1$, the bound involves an additional logarithmic factor.
We adapt their proof to obtain the following quantitative operator-valued Jaffard lemma, which may be of independent interest.

\begin{theorem}
  \label{thm:QI}
  Let $\ch$ be a Hilbert space, $d\in\N$, $s>d$, and choose any $0<\beta<\min\{1,s-d\}$. Then, for every $A\in\J_s(\Z^d;\cb(\ch))$ which is invertible on $\cx=\ell^2(\Z^d;\ch)$, its inverse also belongs to $\J_s(\Z^d;\cb(\ch))$ with
  \begin{align}
    \label{eq:QI}
    \|A^{-1}\|_{\J_s}
    \lesssim_{d,s,\beta} \|A^{-1}\|_{\cb(\cx)} \left(1+\|A\|_{\J_s} \|A^{-1}\|_{\cb(\cx)}\right)^{s/\beta}.
  \end{align}
\end{theorem}

Let us explain the idea of the proof. The decay of $A_{nm}$ for $|n-m|$ large tells us that $A$ is local to some extent, i.e., it interacts rather weakly with distant lattice sites. The boundedness of $A^{-1}$ says that solving $Af=g$ is stable in $\ell^2$. However, considering, e.g., $g=\delta_m\otimes \phi$ ($m\in\Z^d$, $\phi\in\ch$), this does not yet inform us about the profile of $f$, in particular its localization, which we aim to measure in Theorem~\ref{thm:QI}. Our goal is thus to show that a localized $g$ produces a solution $f=A^{-1}g$ whose localization deteriorates polynomially with distance. This goal is similar to elliptic regularity: Invertibility of $1+\mathcal L$ with $\mathcal L=-\sum_{jk}\partial_{x_j}a_{jk}(x)\partial_{x_k}$ (with $a_{jk}$ smooth enough and uniformly elliptic), for example, provides global control of solutions $f$ to $(1+\mathcal L)f=g$ in terms of the data $g$. To upgrade this to higher regularity, one differentiates the equation $(1+\mathcal L)f=g$. In our case, smoothness is replaced with localization, hence we should replace differentiation with localization here. Thus, let $\Psi_k=\diag(\psi_k(n)\one_\ch)_{n\in\Z^d}$ with a bump function $\psi_k$ localizing at $k\in\Z^d$, whose support has diameter on the order of a length scale $L\gtrsim1$ to be determined.
We would then aim for a bound like $\|\Psi_k f\|_2 \lesssim \|A^{-1}\|_{\cb}\|\Psi_k g\|_2$. Such a bound is, however, too optimistic since $A^{-1}$ is nonlocal and this nonlocality produces localization errors $[A,\Psi_k]$ (similarly as in IMS localization \cite{Ismagilov1961,Morgan1979,MorganSimon1981} in many-body quantum mechanics, e.g., in \cite{Sigal1982}).
To control the magnitude of the localization errors, we assume that the $\psi_k$ localize on a rather coarse scale $L$. To return to the original unit scale, we iterate the eventual bound on $\|\Psi_k f\|_2$ and localize it to the original unit cubes. In addition to smallness, the iteration requires measuring the spatial decay of $[A,\Psi_k]$ through estimates on $\|[A,\Psi_k]\Psi_j\|$. Indeed, suppose $V(k-j) \gtrsim \|[A,\Psi_k]\Psi_j\| \|A^{-1}\|$. Then, with the above steps, we have shown $\|\Psi_kf\|\leq \|A^{-1}\|\|\Psi_k g\| + \sum_j V(k-j)\|\Psi_j f\|$. Importantly, we can show $\|V\|_{\ell^1}<1$ if the localization scale $L$ is suitably chosen; it will turn out to depend on both $\|A^{-1}\|_\cb$ and $\|A\|_{\J_s}$. In this case, the localization errors cannot amplify indefinitely, which allows the above-mentioned iteration and eventually return from the coarse-scale blocks to the unit-scale blocks. Let us now give the full

\begin{proof}[Proof of Theorem~\ref{thm:QI}]
  {\bf Step 1: Coarse partition of unity.}
  Let $L\in\{2,3,...\}$ to be chosen later in Step 4. Set $h(t)=(1-|t|)_+$ and, for $k,n\in\Z^d$, let
  \begin{align*}
    \psi_k(n)=\prod_{\nu=1}^dh(n_\nu/L-k_\nu).
  \end{align*}
  Note $\supp\psi_k\subset\{n:\,|n-Lk|_\infty\leq L\}$ and at most $2^d$ of the functions $\psi_k$ are nonzero at any given $n$. Moreover, we have the Lipschitz bound
  \begin{align}
    \label{eq:lip}
    |\psi_k(n)-\psi_k(m)| \leq C_d \min\left\{1,\frac{|n-m|}{L}\right\}.
  \end{align}
  With the $\psi_k$, we introduce
  \begin{align*}
    (\Psi_k)_{mn}=
    \begin{cases}
      \psi_k(n) \one_\ch & \quad \text{if } m=n, \\
      0 & \quad \text{else}.
    \end{cases}
  \end{align*}
  Since $\sum_{k\in\Z} h(t-k)=1$, we have $\sum_k\Psi_k=\one_{\cx}$ strongly. Note that $\sum_k\psi_k(n)^2\leq\sum_k\psi_k(n)=1$, and hence $\sum_k\|\Psi_kf\|_{\cx}^2\leq\|f\|_\cx^2$.

  \medskip
  \noindent
  {\bf Step 2: Magnitude of the localization error.}
  Set
  $$
  a:=\|A\|_{\J_s}.
  $$
  Since $[A,\Psi_k]_{mn}=A_{mn}(\psi_k(n)-\psi_k(m))$, we have, by \eqref{eq:lip} and $\min\{1,t\}\leq t^\beta$ for all $0<\beta\leq1$,
  $$
  \|[A,\Psi_k]_{mn}\|_{\cb(\ch)} \leq C_d a L^{-\beta} |n-m|^\beta (1+|n-m|)^{-s}, \quad 0<\beta\leq1.
  $$
  Since $s-\beta>d$, the scalar row and column sums of $\|[A,\Psi_k]_{nm}\|$ are bounded uniformly in $k$. Hence,
  \begin{align}
    \label{eq:comm}
    \|[A,\Psi_k]\|_{\cb(\cx)} \leq C_{d,s,\beta}aL^{-\beta}, \quad 0<\beta\leq1
  \end{align}
  follows by an operator-valued Schur test \cite[Lemma~3.3]{KohldorferBalazs2025W}; the latter states for $T=(T_{mn})_{m,n\in\Z^d}$ with $T_{mn}\in\cb(\ch)$ that
  $$
  \|T\|_{\cb(\cx)} \leq \left(\sup_m\sum_n\|T_{mn}\|_{\cb(\ch)}\right)^{1/2} \left(\sup_n\sum_m\|T_{mn}\|_{\cb(\ch)}\right)^{1/2},
  $$
  which follows from applying the scalar Schur test to the matrix of operator norms $(\|T_{mn}\|_{\cb(\ch)})_{m,n}$.

  \medskip
  \noindent
  {\bf Step 3: Decay of localization errors.}
  Let $P_j$ denote the orthogonal projection onto $\supp\psi_j$. If $|k-j|_\infty>4$, then the supports of $\psi_k$ and $\psi_j$ are disjoint, i.e., $\Psi_kP_j=0$ and $[A,\Psi_k]P_j=-\Psi_kAP_j$. Considering $[A,\Psi_k]P_j$, then every input site $n\in\supp\psi_j$ and output site $m\in\supp\psi_k$ satisfy $|n-m|\geq c_d L(1+|k-j|)$. There are at most $C_dL^d$ sites in either support. Thus, by an operator-valued Schur test,
  \begin{align}
    \|[A,\Psi_k]P_j\|_{\cb(\cx)}
    = \|\Psi_k A P_j\|_{\cb(\cx)}
    \leq C_{ds} a L^{d-s} (1+|k-j|)^{-s}, \quad |k-j|_\infty > 4.
  \end{align}
  On the other hand, by \eqref{eq:comm},
  \begin{align*}
    \|[A,\Psi_k]P_j\|_{\cb(\cx)} \leq C a L^{-\beta}, \quad |k-j|_\infty \leq 4.
  \end{align*}

  \medskip
  \noindent
  {\bf Step 4: Localization of the solution on a coarse block.}
  We now control the solution $f$ of $Af=g$ when localized using $\Psi_k$. Let us abbreviate
  $$
  B=A^{-1},
  \qquad
  b=\|B\|_{\cb(\cx)}
  $$
  so that $f=Bg$ and set
  $$
  u_k = \|\Psi_k f\|_{\cx}, \qquad v_k = b\|\Psi_kg\|_{\cx} = b\|\Psi_kAf\|_{\cx}.
  $$
  Since $\Psi_kf=BA\Psi_kf=B\left(\Psi_k Af + [A,\Psi_k]f\right)$, we have, by the triangle inequality,
  \begin{align}
    \label{eq:bound-uk}
    u_k \leq b\|A\Psi_k f\|
    \leq b\|\Psi_kAf\| + b\|[A,\Psi_k]f\|
    \leq v_k + b \sum_j\|[A,\Psi_k]\Psi_jf\|.
  \end{align}
  We claim that the last sum converges absolutely in norm for a suitably chosen $L$. Indeed, since $\Psi_jf$ is supported in $P_j\cx$, we have, by Step~3,
  \begin{align}
    \label{eq:scalarineq}
    u_k \leq v_k + \sum_{j\in\Z^d}V(k-j)u_j,
  \end{align}
  where $V:\Z^d\to[0,\infty)$ is the translation invariant majorant
  \begin{align}
    \label{eq:V}
    V(h) := C_0ab
    \begin{cases}
      L^{-\beta}, & \quad |h|_\infty\leq4, \\
      L^{d-s}(1+|h|)^{-s}, & \quad |h|_\infty > 4.
    \end{cases}
  \end{align}
  Thus, $\|V\|_{\ell^1} \lesssim ab L^{-\beta}$ since $\beta<s-d$. Since $u=(u_k)_{k\in\Z^d}\in\ell^2(\Z^d)$, we see that the sum in~\eqref{eq:bound-uk} indeed converges absolutely in norm. If we choose $L\in\{2,3,...\}$ large enough with
  \begin{align}
    \label{eq:L}
    L \sim_{d,s,\beta} (1+ab)^{1/\beta},
  \end{align}
  then
  \begin{align*}
    \|V\|_{\ell^1} \leq \frac12,
    \qquad 
    \sup_h(1+|h|)^s V(h) \lesssim ab L^{-\beta} \lesssim1.
  \end{align*}
  By Step~1, $u=(u_k)_{k\in\Z^d}$ and $v=(v_k)_{k\in\Z^d}$ are in $\ell^2(\Z^d)$. Iterating \eqref{eq:scalarineq} $N$ times yields
  \begin{align}
    u \leq \left(\sum_{n=0}^{N-1} V^{*n}\right) \ast v + V^{*N} \ast u
  \end{align}
  componentwise. Here, we write $V^{*n}$ for the $n$-fold convolution, defined iteratively by
  $$
  V^{*n}(k) = \left(V\ast V^{*(n-1)}\right)(k) = \sum_{\ell\in\Z^d}V(k-\ell) V^{*(n-1)}(\ell)
  $$
  with the conventions $V^{*1}=V$ and $V^{*0}=\delta_0$. Since $\|V^{*N}\ast u\|_{\ell^2}\leq 2^{-N}\|u\|_2$,
  \begin{align}
    \label{eq:Wv}
    u_k \leq \sum_j W(k-j)v_j, \qquad W=\sum_{n=0}^\infty V^{*n}.
  \end{align}
  In Lemma~\ref{lem:renewal} below, we show
  \begin{align}
    \label{eq:W}
    0\leq W(h) \leq C_{d,s,\beta} (1+|h|)^{-s}, \quad h\in\Z^d.
  \end{align}

  \medskip
  \noindent
  {\bf Step 5: From coarse blocks to unit blocks.}
  We now prove the claimed estimate for the matrix elements $B_{nm}$. Let $g$ be supported on a single site $m\in\Z^d$. Fix an output site $n\in\Z^d$. Then there are at most $2^d$ coarse indices $j$ with $\psi_j(m)\neq0$; likewise there are at most $2^d$ many coarse indices $k$ with $\psi_k(n)\neq0$. Decomposing $(Bg)_n=f_n=\sum_k\psi_kf_n$, we obtain, by \eqref{eq:Wv}--\eqref{eq:W},
  \begin{align}
    \|(Bg)_n\|_\ch
    \leq \sum_{k:\,\psi_k(n)\neq0}u_k
    \leq C b \sum_{\substack{k:\,\psi_k(n)\neq0 \\ j:\,\psi_j(m)\neq0}}(1+|k-j|)^{-s}\|g\|_{\ch}.
  \end{align}
  For such $k,j$ in the above sum, we have $|n-Lk|,|m-Lj|\leq C_d L$, so that $1+|k-j|\geq c_d(1+|n-m|/L)$. Hence,
  \begin{align}
    \|B_{nm}\| \lesssim_{d,s,\beta} b (1+|n-m|/L)^{-s} \lesssim b L^s (1+|n-m|)^{-s},
  \end{align}
  which was to be proved.
\end{proof}

We finally provide a lemma which implies \eqref{eq:W} above.

\begin{lemma}
  \label{lem:renewal}
  Let $s>0$ and $v:\Z^d\to[0,\infty)$ satisfy $\rho:=\|v\|_{\ell^1}<1$ and $D:=\sup_{h\in\Z^d}(1+|h|)^sv(h)<\infty$. Let $W=\sum_{n=0}^\infty v^{*n}$. Then,
  \begin{align}
    \label{eq:renewal}
    \sup_h(1+|h|)^sW(h) \leq 1+D\sum_{n=1}^\infty n^{s+1}\rho^{n-1}.
  \end{align}
  In particular, the right-hand side is $\lesssim_s 1+D$ if $\rho\leq1/2$.
\end{lemma}

\begin{proof}
  Let us consider the $n$-th summand, i.e.,
  $$
  v^{*n}(h) = \sum_{\substack{h_1,...,h_n\in\Z^d \\ h_1+...+h_n=h}}v(h_1)v(h_2)...v(h_n).
  $$
  For $h\in\Z^d$ with $h=h_1+...+h_n$ there is, by pigeonholing, an index $j\in\{1,...,n\}$ such that $|h_j|\geq|h|/n$. For this $j$, we have
  \begin{align}
    \label{eq:renewal-aux1}
    v(h_j) \leq D(1+|h_j|)^{-s} \leq D n^s (1+|h|)^{-s}.
  \end{align}
  We now split the above $n$-fold sum according to the largest $|h_j|$. To this end, let
  \begin{align*}
    \mathcal A_j = \left\{(h_1,...,h_n):\, \sum_{i=1}^n h_i = h, \ |h_j|\geq\frac{|h|}{n}\right\}, \quad j\in\{1,...,n\}.
  \end{align*}
  Every tuple $(h_1,...,h_n)$ satisfying $\sum_{i=1}^n h_i=h$ belongs to at least one $\mathcal A_j$. Since $v\geq0$, we can therefore bound
  \begin{align*}
    v^{*n}(h) \leq \sum_{j=1}^n \sum_{(h_1,...,h_n)\in\mathcal A_j}\prod_{\ell=1}^n v(h_\ell).
  \end{align*}
  For the $j$-th summand, we have, by \eqref{eq:renewal-aux1},
  \begin{align*}
    \sum_{(h_1,...,h_n)\in\mathcal A_j}\prod_{\ell=1}^n v(h_\ell)
    \leq D n^s (1+|h|)^{-s} \sum_{(h_1,...,h_n)\in\mathcal A_j}\prod_{\ell\in\{1,...,n\}\setminus\{j\}} v(h_\ell).
  \end{align*}
  Since $h_j=h-\sum_{\ell\neq j}h_\ell$ and $v(h_j)$ is removed from the sum over $n$-tuples, we can sum freely over the other $n-1$ variables to get
  \begin{align*}
    \begin{split}
      \sum_{(h_1,...,h_n)\in\mathcal A_j}\prod_{\ell\in\{1,...,n\}\setminus\{j\}} v(h_\ell)
      & \leq \sum_{(h_k)_{k\neq j}\in(\Z^d)^{n-1}} \prod_{\ell\in\{1,...,n\}\setminus\{j\}} v(h_\ell) \\
      & = \prod_{\ell\neq j}\left(\sum_{k\in\Z^d}v(k)\right) = \rho^{n-1}.
    \end{split}
  \end{align*}
  Hence, we have shown
  \begin{align*}
    \sum_{(h_1,...,h_n)\in\mathcal A_j}\prod_{\ell=1}^n v(h_\ell)
    \leq D n^s (1+|h|)^{-s} \rho^{n-1}
  \end{align*}
  and thus
  \begin{align*}
    v^{*n}(h)
    \leq \sum_{j=1}^n D n^s (1+|h|)^{-s} \rho^{n-1}
    = Dn^{s+1}\rho^{n-1}(1+|h|)^{-s}.
  \end{align*}
  Summing over $n$ concludes the proof of \eqref{eq:renewal}.
\end{proof}


\section{Resolvent calculus and localized smoothing estimates for compactly supported multipliers}
\label{s:hs-calculus}

The goal of this section is to prove \eqref{eq:window-intro} following the strategy outlined in Subsection~\ref{ss:proofsketch}. We will reduce the proof of \eqref{eq:window-intro} to controlling $\one_{Q_a}(z-K)^{-1}\one_{Q_b}$ with $K=(H+1)^{-1}$ in $\J_{d+\alpha}$. To this end, we first record
\begin{align*}
  K=(H+1)^{-1}\in\J_{d+\alpha}, \qquad \|K\|_{2\to2}\leq1.
\end{align*}
Hence,
$$
\|u+iv-K\|_{\J_{d+\alpha}} \leq C, \qquad \|(u+iv-K)^{-1}\|_{2\to2}\leq|v|^{-1}, \quad u\in[-1,2], \quad 0<|v|\leq1.
$$
Thus, by Theorem~\ref{thm:QI}, with
\begin{align}
  \label{eq:indices2}
  \beta_* = \min\{1,\alpha\}/2, \qquad
  M_* = 1+(d+\alpha)/\beta_*,
\end{align}
we obtain,
\begin{align}
  \label{eq:res}
  \|(u+iv-K)^{-1}\|_{\J_{d+\alpha}} \leq C |v|^{-M_*}, \quad u\in[-1,2], \quad 0<|v|\leq1,
\end{align}
which is precisely \eqref{eq:key-intro2}. Using \eqref{eq:res} and the Helffer--Sj\"ostrand formula, we can then prove \eqref{eq:key-intro}, i.e., the following

\begin{lemma}
  \label{lem:HS}
  Let $J\subset(0,1)$ be compact. Then for every $L\in\N$ with $L>M_*$ there is $C_{J,L}<\infty$ such that
  $$
  \|\phi(K)\|_{\J_{d+\alpha}} \leq C_{J,L}\|\phi\|_{C^{L+1}(\R)}, \quad \phi\in C_c^\infty(J).
  $$
\end{lemma}

\begin{proof}
  Let $\eta\in C_c^\infty((-1,1);[0,1])$ be a cutoff with $\eta(y)=1$ for $|y|\leq1/2$. We use this to construct the almost-analytic extension of $\phi$ at order $L$,
  $$
  \widetilde\phi(x+iy) = \eta(y)\sum_{j=0}^L\frac{(iy)^j\phi^{(j)}(x)}{j!},
  $$
  which is zero unless $|y|\leq1$ and $x$ is in a fixed compact subinterval of $(0,1)$. Then,
  $$
  |\bar\partial\widetilde\phi(x+iy)| \leq C_{J,L} \|\phi\|_{C^{L+1}} \begin{cases} |y|^L, & \quad |y|\leq1/2, \\ 1, & \quad |y|\geq1/2.
  \end{cases}
  $$
  We now claim that the Helffer--Sj\"ostrand formula
  \begin{align}
    \label{eq:helffersjostrand}
    \phi(K) = \frac{1}{2\pi i} \int_\C \bar\partial\widetilde\phi(z) (z-K)^{-1}\,dz\wedge d\bar z
  \end{align}
  is well-defined as an absolutely convergent Bochner integral in the Banach space $\J_{d+\alpha}$. Indeed, for $z,w\in\C\setminus\R$ in a compact subset of the resolvent set of $K$, the resolvent formula
  $$
  (z-K)^{-1}-(w-K)^{-1} = (w-z)(z-K)^{-1}(w-K)^{-1}
  $$
  and \eqref{eq:res}, together with continuity of multiplication in $\J_{d+\alpha}$, imply local $\J_{d+\alpha}$-norm continuity of $(z-K)^{-1}$ on the nonreal integration region of \eqref{eq:helffersjostrand}. The integrand $\bar\partial\widetilde\phi(z) (z-K)^{-1}$, defined arbitrarily on the Lebesgue-measure-zero line $\im z=0$, is therefore strongly Bochner measurable. The preceding estimates for $\bar\partial\widetilde\phi$ and $L>M_*-1$ thus make \eqref{eq:helffersjostrand} indeed an absolutely convergent Bochner integral in $\J_{d+\alpha}$. The continuous inclusion $\J_{d+\alpha}\hookrightarrow\cb(L^2)$ identifies the right-hand side of \eqref{eq:helffersjostrand} with $\phi(K)$ defined by the $L^2$-functional calculus.
\end{proof}

With Lemma~\ref{lem:HS} at hand, we can now show \eqref{eq:window-intro}. Recall from \eqref{eq:indices}, $S_*=\lceil M_*\rceil+5$ with $M_*$ defined as in \eqref{eq:indices2}.

\begin{theorem}
  \label{thm:window}
  Let $I\subset(0,\infty)$ be a compact interval, $s>S_*$, $\sigma>0$, $\rho=\sigma^{1/\alpha}$, and $m\in H^s(\R)$ be supported in $I$. Then for every $z\in\R^d$ and $j\in\N_0$,
  \begin{align}
    \label{eq:window}
    \|\one_{B(z,\rho)}m(\sigma H)\one_{A_j(z,\rho)}\|_{2\to q}
    \lesssim_{I,s,q} \rho^{-d(1/2-1/q)} 2^{-j(d/2+\alpha)} \|m\|_{H^s}.
  \end{align}  
  The same holds with $m(\sigma H)^*$ in place of $m(\sigma H)$.
\end{theorem}

\begin{proof}
  First, by the homogeneity of $H$, it suffices to consider $\rho=\sigma=1$. Moreover, we will just prove the bound for $m(H)$ as the proof for the adjoint follows from $m(H)^*=\overline m(H)$ and $\|\overline m\|_{H^s}=\|m\|_{H^s}$. Furthermore, since $s>S_*$, we have $H^s(\R)\hookrightarrow C^{L+1}(\R)$ for $L=\lceil M_*\rceil+2$. Hence, by an approximation argument, it suffices to consider $m\in C^{L+1}(\R)$, or even just $C^\infty(\R)$.

  \medskip
  \noindent
  {\bf Step 1: Finitely many small $j$.}
  For finitely many small $j$, \eqref{eq:window} follows from the factorization $m(H)=\me{-H}(\me{H}m(H))$, the bound $\|\me{H}m(H)\|\lesssim_I\|m\|_\infty$, and the $L^2\to L^q$-boundedness of $\me{-H}$ (Proposition~\ref{prop:cubes}). Indeed, we then get
  \begin{align*}
    \|m(H)\|_{2\to q} \lesssim_I \|m\|_\infty \lesssim_{I,s}\|m\|_{H^s}.
  \end{align*}
  Thus, it remains to consider $j\gtrsim1$.

  \medskip
  \noindent
  {\bf Step 2: Factorization of $m$.}
  Let $K=(H+1)^{-1}$. To apply Lemma~\ref{lem:HS}, we factorize $m(H)=\me{-H}\phi(K)$ with
  $$
  \phi(\lambda) =
  \begin{cases}
    \exp(\lambda^{-1}-1) m(\lambda^{-1}-1), & \quad \lambda\in(0,1], \\
    0, & \quad \lambda=0.
  \end{cases}
  $$
  If $m\in C_c^\infty(I)$, then $\phi\in C_c^\infty(J)$ for a fixed compact $J\subset (0,1)$ determined by $I$. The map $m\mapsto\phi$ is smooth on the relevant compact sets, and $\|\phi\|_{C^{L+1}}\lesssim_I \|m\|_{C^{L+1}}$. Hence, by Lemma~\ref{lem:HS},
  \begin{align}
    \label{eq:JC}
    \|\phi(K)\|_{\J_{d+\alpha}} \leq C_I \|m\|_{C^{L+1}}.
  \end{align}

  \medskip
  \noindent
  {\bf Step 3: Localized smoothing bound for $m$.}
  With the above decomposition of $m$ and the localized $L^2\to L^q$-bounds for $\me{-H}$ in Proposition~\ref{prop:cubes}, we obtain
  \begin{align}
    \label{eq:window-cubes}
    \begin{split}
      \|\one_{Q_a}m(H)\one_{Q_b}\|_{2\to q}
      & \leq \sum_{c\in\Z^d} \|\one_{Q_a}\me{-H}\one_{Q_c}\|_{2\to q} \|\one_{Q_c}\phi(K)\one_{Q_b}\|_{2\to2} \\
      & \lesssim \|\phi(K)\|_{\J_{d+\alpha}} \sum_{c\in\Z^d}(1+|a-c|)^{-d-\alpha}(1+|c-b|)^{-d-\alpha} \\
      & \lesssim \|m\|_{C^{L+1}} (1+|a-b|)^{-d-\alpha}.
    \end{split}
  \end{align}

  \medskip
  \noindent
  {\bf Step 4: From cubes to annuli.}
  We now cover $B(z,1)$ and $A_j(z,1)$ with unit cubes. Clearly, $B(z,1)$ intersects at most $C_d$ many unit cubes, which we collect in a family. Likewise, $A_j(z,1)$ intersects at most $C_d2^{jd}$ many unit cubes, which we collect in another family, say $\mathcal B_j$. We use these two families of unit cubes to cover $B(z,1)$ and $A_j(z,1)$, respectively. Since we can assume $j\gtrsim1$ to be sufficiently large by Step~1, the distance between the cubes of the two families is of order $2^j$. Hence, by \eqref{eq:window-cubes} and Cauchy--Schwarz,
  \begin{align*}
    \|\one_{B(z,1)}m(H)\one_{A_j(z,1)}f\|_q
    & \lesssim \|m\|_{H^s} 2^{-j(d+\alpha)} \sum_{b\in\mathcal B_j}\|\one_{Q_b}\one_{A_j(z,1)}f\|_2 \\
    & \lesssim \|m\|_{H^s} 2^{-j(d+\alpha)}2^{jd/2}\|f\|_2.
  \end{align*}
  This concludes the proof.
\end{proof}

\section*{Acknowledgments}
I wish to thank Volker Bach, The Anh Bui, Jean-Claude Cuenin, and Atsuhide Ishida for valuable discussions.


\newcommand{\etalchar}[1]{$^{#1}$}
\def\cprime{$'$}

\end{document}